\documentclass[twoside,12pt]{amsart}

\usepackage{amsmath}
\usepackage{amsthm}
\usepackage{amsfonts}
\usepackage{latexsym}

\usepackage{upref}
\usepackage{caption,subcaption} 
\usepackage{float}

\usepackage{mathptmx}
\usepackage{graphicx}
\usepackage{txfonts}
\usepackage[mathscr]{eucal}
\usepackage{nicefrac}
\usepackage{dsfont}
\usepackage{cancel}

\makeatletter
\def\serieslogo@{}
\makeatother \makeatletter
\def\@setcopyright{}

\usepackage{color}

\def\div{div\,}

\def\r{{{r}}} 

\def\bu{{\bf u}}
\def\bv{{\bf v}}
\def\bw{{\bf w}}

\def\lam{\lambda}
\def\u{{\mathbf u}}  
\def\z{{\mathbf z}}   

\def\Leb{L}

\newcommand{\RR}{\mathbb R}
\newcommand{\TT}{\mathbb T}
\newcommand{\Om}{\Omega}
\newcommand{\Omd}{\Om} 

\renewcommand{\geq}{\geqslant}
\renewcommand{\leq}{\leqslant}

\def\rp{2} 
\def\Frechet{Fr\'{e}chet }

\newcommand{\D}{\textnormal{d}}
\newcommand{\dt}{{\D}t}
\newcommand{\dds}{{\D}\!s}
\newcommand{\ds}{\displaystyle}

\newcommand{\B}{\Leb^\infty} 
\newcommand{\Bnorm}[1]{ {\|#1\|_{\B}}}

\newcommand{\JT}{{\mathcal V}}
\def\weakLp{\Leb^{p,\infty}}
\newcommand{\wkLp}[1]{\|#1\|^p_{\weakLp}}

\def\weakLp{\Leb^{p,\infty}}

\newcommand\dx{{\D}\hspace*{-0.02cm}x}
\newcommand\dy{{\D}\hspace*{-0.005cm}y}
\newcommand\dmu{{\D}\mu}

\newcommand\Fvee{{\JT}}

\newcommand{\be}{\begin{equation}}
\newcommand{\ee}{\end{equation}}
\def\cD{{\mathcal D}}
\newcommand\eref[1]{{\rm (\ref{#1})}}

\newcommand{\iref}[1]{{\rm (\ref{#1})}}

\renewcommand{\geq}{\geqslant}
\renewcommand{\ge}{\geqslant}
\renewcommand{\leq}{\leqslant}
\renewcommand{\le}{\leqslant}

\theoremstyle{plain}
\newtheorem{THEOREM}{Theorem}[section]

\newtheorem{theorem}[THEOREM]{Theorem}

\newtheorem*{assumption}{Assumption}

\theoremstyle{definition}
\newtheorem{lemma}{Lemma} 

\theoremstyle{remark}

\newtheorem{remark}[THEOREM]{Remark}

\newtheorem{example}[THEOREM]{Example}

\numberwithin{equation}{section}

\begin{document}

\title[Constructions of bounded solutions of $\div \bu=f$ in critical  spaces]{construction of bounded solutions of $\div \bu=f$ in critical  spaces}

\author[Albert Cohen]{Albert Cohen} 
\address[Albert Cohen] {
\newline Laboratoire Jacques-Louis Lions
\newline UPMC Univ Paris 06, UMR 7598 
\newline Paris, France F-75005}
\email[]{cohen@ann.jussieu.fr}

\author[Ronald DeVore]{Ronald DeVore} 
\address[Ronald DeVore] {
\newline Department of Mathematics
\newline Texas A\&M University
\newline College Station, TX 77840, USA}
\email[]{rdevore@math.tamu.edu}

\author[Eitan Tadmor]{Eitan Tadmor}
\address[Eitan Tadmor]{
\newline Department of Mathematics
\newline Institute for Physical sciences and Technology (IPST)
\newline University of Maryland
\newline MD 20742-4015, USA}
\email[]{tadmor@umd.edu}

\date{\today}


\thanks{\textbf{Acknowledgment.} 
This research was supported by the NSF grant DMS-2134077 of the MoDL program as well as the MURI ONR grant N00014-20-1-2787, ONR grant N00014-21-1-2773 and the Fondation Sciences Math\'ematiques de Paris.  }

\setcounter{page}{1}

\begin{abstract} 
We construct uniformly bounded solutions of the equation $\div \bu=f$  for arbitrary data  $f$ in the critical  spaces $ L^d(\Omega)$, 
where $\Omega$ is a domain of $\RR^d$.
This question was addressed by Bourgain \& Brezis, \cite{BB2003}, who proved that although the problem has a uniformly bounded solution, it is critical in the sense that there exists no linear solution operator for general $L^d$-data. We first discuss the validity of this existence 
result under weaker conditions than $f\in L^d(\Omega)$, and then 
focus our work on constructive processes for such uniformly bounded solutions.
In the $d=2$ case,  we present a direct one-step explicit construction,  
which generalizes for $d>2$ to a $(d-1)$-step construction based on induction. 
An explicit construction is proposed for compactly supported 
data in $L^{2,\infty}(\Omega)$ in the $d=2$ case.
We also present constructive approaches based on optimization
of a certain loss functional adapted to the problem.  This approach provides a two-step construction in the $d=2$ case. This
optimization is used as the building block of a hierarchical multistep process 
introduced in \cite{Tad2014} that
converges to a solution in more general situations. 
\end{abstract}

\maketitle

\centerline{\it Dedicated to our friend and colleague Wolfgang Dahmen on his 75th birthday}

\setcounter{tocdepth}{2}
{\small \tableofcontents}

\section{Introduction}

Let $Y$ denote a Banach space of functions defined on a $d$-dimensional
domain $\Omega \subset \RR^d$, where $d\geq 2$.
We are concerned with the existence and  construction of uniformly bounded solutions $u$ to the equation, 
\begin{equation}
\div \bu=f,
\label{diveq}
\end{equation}
whenever $f\in Y$.  
Namely,   we ask whether there exists a  $\gamma>0$ such that for every
$f\in Y$ there exists a solution  $\bu=(u_1,\dots,u_d)\in L^\infty(\Omd)$  to \iref{diveq}
such that
\begin{equation}\label{eq:divU=F}
\|{\bf u}\|_{L^\infty} \leq \gamma \|f\|_{Y}.
\end{equation}
Here for a vector valued function $\bv=(v_1,\dots,v_d)$ such that each $v_i$ belongs to a function space $X$, we use the simpler notation $\bv\in X$ and $\|\bv\|_X$ instead of $X^d$. We say the space  $Y$ is {\it admissible}
if for all $f\in Y$, \eqref{diveq} admits a solution such that \eref{eq:divU=F} holds.

There exists of course infinitely many solutions to \iref{diveq} since
as soon as one exists, we can add to it a null divergence function, for example a constant. The most natural  candidate for a solution $\bu$ when given $f$ is 
to solve Laplace's  equation with data $f$ and then to take $\bu$ as the gradient
of the solution. More precisely, we introduce
$$
\psi(x)=\nabla \phi(x)= \frac {C_d}{|x|^{d}}x,\quad x\in\RR^d,
$$
where $\phi$ is the fundamental solution of the Laplacian on $\RR^d$, and define
the so-called Helmholtz solution as
$$
\bu(x)=\bu_{\text{Hel}}(x)=\int_\Omd f(x)\psi(x-y)\dy= \tilde f *\psi(x),
$$
where $\tilde f(x)=f(x)$ for $x\in \Omd$ and $\tilde f(x)=0$ when $x\notin \Omd$.
Note that $\bu_{\text{Hel}}$ depends linearly on $f$.
When $\Omega$ is a bounded domain, it is readily seen that $\bu_{\text{Hel}}$ is a uniformly bounded solution of \eqref{diveq} for $f\in L^p(\Omd)$ whenever  $p>d$, and therefore 
the spaces $Y=L^p$, $p>d$ are all admissible. 

The question of whether  $Y=L^d$ is admissible was addressed in the seminal work of Bourgain \& Brezis \cite{BB2003}.
Their work studies the particular case where $\Omega=\TT^d$ is the $d$-dimensional torus, which leads to
assume in addition that $\int_{\TT^d} f=0$.  They proved that the problem \eqref{diveq} is
 \emph{critical} in the sense that it admits bounded solutions, but there is no \emph{linear} 
 solution operator from $Y$ to $L^\infty$. In particular, one cannot invoke the Helmholtz solution.
 We say a space $Y$ is critical if it is admissible but there is no linear mapping taking $f\in Y$ into a solution $\bu\in L^\infty$ of \eref{eq:divU=F}.
 
 The main interest of the present paper is two-fold.   We first ask   which of the classical function spaces  $Y$ are admissible.  Secondly, we are interested in  explicit constructions of solutions to the Bourgain-Brezis problem.  In particular, can we explicitly construct nonlinear mappings solving \eref{eq:divU=F} when $Y$ is 
critical. In section \ref{sec:theory} we discuss theoretical aspects of the problem.  We recall certain known results of Meyer which show that for $\Omega=\RR^d$, 
the space $G$ of all $f$ that admit a solution  $\bu\in L^\infty$ to \eref{eq:divU=F} is the dual space
$W^{1,1}_{\rm hom}$ which is defined as the closure
of the smooth test functions for the total variation. 
Therefore, any admissible space $Y$ must be a subspace of $G$.   In particular, $Y$ is admissible
if and only if $W^{1,1}_{\rm hom}$ embeds into $Y^*$.    
In particular, we show that not only is $L^ d$  admissible but also the larger space weak-$L^d$, i.e. $L^{d,\infty}(\Omd)$,  is admissible, as well as even weaker Morrey spaces.

In section \ref{sec:bounded}, we present explicit constructions
of bounded solutions for certain critical admissible spaces $Y$.   We give a one step formula in the $d=2$ case with $L^2$-data,
and we treat the case $d>2$ with $L^d$-data
by a $(d-1)$-step construction based on induction.  
We end this section by a construction for $L^{2,\infty}$-data in the case $d=2$
assuming in addition compact support.
The reader may find these constructions  interesting for their own sake. 
In section \ref{sec:variat}, we propose variational-based approaches for the
constructions of bounded solutions. In the case  $d=2$, this approach 
delivers the solution in two steps. More generally, we use this optimization
as the building block of a hierarchical decomposition 
that was used in \cite{Tad2014} to construct solutions to the Bourgain-Brezis 
with $L^d$ data by a limiting process. We use this 
multi-step hierarchical approach to construct solutions 
for more general data.

\section{Theory}\label{sec:theory}

\subsection{Existence of bounded solutions for $L^{d}$-data}

{\bf Let $\Omega\subset\RR^d$.} The space  
$$
G=G(\Omega):=\{f=\div(\bu) \; :\; \bu\in L^\infty(\Omega)\},
$$ 
of distributions $u$  whose divergence is uniformly bounded has been
studied in various contexts, in particular image processing and nonlinear PDE's.
 
As noted in \cite{Mey2002}, for the case $\Omega=\RR^d$,  
the space $G(\RR^d)$ is the dual of the homogeneous space $W^{1,1}_{\rm hom}(\RR^d)$. The latter is defined as is the completion of the space of test functions 
$\cD(\RR^d)$ for the total variation, which defines a norm on this space.  

Let us recall that the total variation of  $v\in BV(\Omega)$ is defined as
$$
|v|_{TV}:=\sup _\bw \int_\Omega v \,\div \bw,
$$
where the supremum is taken over all $\bw=(w_1,\dots,w_d)\in {\mathcal D}(\Omega)$
such that $\|\bw\|_{L^\infty}=\sup_{x\in \Omega} |\bw(x)|_2\leq 1$. An equivalent 
quantity is defined in terms of finite difference:
$$
|v|_{TV}\sim \sup_{h>0} h^{-1}\sup_{|y|\leq h} \|v-v(\cdot-y)\|_{L^1(\Omega_h)},
$$
where $\Omega_h:=\{x\in \Omega\; : \; {\rm dist}(x,\partial\Omega)>h\}$.
When $\nabla v\in L^1$, in particular when $v\in W^{1,1}_{\rm hom}$, one simply has
$$
|v|_{TV}=\|\nabla v\|_{L^1}.
$$
Let us stress that $W^{1,1}_{\rm hom}(\RR^d)$ is strictly smaller than $BV(\RR^d)$.

Therefore, every $f$ in a function space $Y$
of locally integrable functions defined on $\RR^d$ admits the representation $f= \div \bu$ 
with a uniformly bounded $\bu$ satisfying the bound \eqref{eq:divU=F}
if and only if for any test function $g\in {\mathcal D}(\RR^d)$ one has
\begin{equation}
\Big |\int_{\RR^d} f g \Big |\leq \gamma \|f\|_Y|g|_{TV}.
\label{eq:apriori}
\end{equation}
Note that this is equivalent
to the condition that
\begin{equation}
\Big |\int_{E} f\Big | \leq \gamma \|f\|_Y{\rm per}(E),
\label{eq:aprioriper}
\end{equation}
for all open sets $E$ of finite perimeter. Indeed the above is obtained from \iref{eq:apriori}
by taking $g=\varphi_\epsilon*\chi_E$ where $\varphi_\epsilon$
is a mollifier and letting $\epsilon\to 0$.
But from the 
coarea formula 
\begin{equation}
|g|_{TV}=\int_{-\infty}^{+\infty} {\rm per}(E_t)\dt, \quad E_t:=\{x \; :\; g(x)>t\},
\end{equation}
see \cite{EG1992}, it also implies \iref{eq:apriori} for any $g\in \cD(\RR^d)$. Note that 
here, the perimeter ${\rm per}(E)$ coincides with the Hausdorff measure
${\mathcal H}^{d-1}(\partial E)$ only
for sufficiently nice sets (for example with Jordan domains with rectifiable boundaries). 
More generally it should be be defined as $|\chi_E|_{TV}$ or equivalently 
as ${\mathcal H}^{d-1}(\partial E^*)$ where $\partial E^*$ is the so-called reduced 
boundary as introduced by de Giorgi. 

\begin{remark} The co-area formula also shows that 
\iref{eq:aprioriper} actually implies the validity of 
\iref{eq:apriori} for any $g\in BV(\RR^d)$. Therefore any $f\in G$ that is
in addition locally integrable is also an element of the dual of $BV$. We give
further in Remark \ref{remarkdual} an example of a distribution that belongs to $G$ but are not 
in the dual of $BV$.
\end{remark}

For $Y=L^d(\RR^d)$, 
the validity of \eqref{eq:apriori} is ensured by the Sobolev embedding of $BV(\RR^d)$ 
into $L^{d'}(\RR^d)$ where $\frac 1 d+\frac 1 {d'}=1$. Since for a general domain $\Omega$, we can 
trivially extend any $f\in L^d(\Omega)$ by $0$
to obtain a function of $L^d(\RR^d)$ with  the same norm, this 
implies that  $Y=L^d(\Omega)$ is admissible. 

In  \cite{BB2003}, the same result is given in the periodic context, 
where $\Omega=\TT^d$ is the $d$-dimensional torus. In this case, $Y=L^d(\Omega)$ is modified into
$$
Y=L^d_{\#}(\TT^d)=\left\{f\in L^d(\TT^d)\; : \; \int_\Omd f=0\right\}
$$
As shown in Proposition 2 therein, there exists no linear solution operator 
$f\in L^d_{\#}(\TT^d) \mapsto \bu\in L^\infty(\TT^d)$. Indeed,  
restricting attention to the simpler case of the two-dimensional torus, 
if $K: L^2_{\#}(\TT^2) \mapsto L^\infty(\TT^2)$ would be such a linear solution operator so  that $\div K={\mathbb I}$ is the identity, then so is 
$$
\ds \widetilde{K}:=\int \limits_{y\in \TT^2} \tau_{-y} K\tau_y \dy,
$$ 
which averages $K$ over all 2D translations $\tau_y$. Now $\widetilde{K}$ is translation invariant and it has a symbol,
$\Lambda(n):=(\lambda_1(n),\lambda_2(n))$ such that 
$\widetilde{K}(e^{in\cdot x})=\Lambda(n) e^{i n\cdot x}$. Since $\widetilde{K}$ is assumed to boundedly map $L^2$ to $L^\infty$, one should have $(\Lambda(n))_{n\in {\mathbb Z}^2}
\in \ell^2({\mathbb Z}^2)$. However, since $\div \widetilde{K}=\div K={\mathbb I}$, that is $n\cdot \Lambda(n)=1$, this implies $\ds |\Lambda(n)|_{{}_2} \geq \frac{1}{|n|_{{}_2}}$ which is
a contradiction to  $(\Lambda(n))_{n\in {\mathbb Z}^2}
\in \ell^2({\mathbb Z}^2)$.  

The lack of linearity is attributed to the general fact that the problem of solving  ${\mathcal L}\bu = f$ with 
$\bu\in X$ is critical if $\textnormal{Ker}({\mathcal L})$ has no complement in $X$ \cite{BB2007,Aji2009}.
This is the case of $\div$ in $L^\infty$. One of the main themes in \cite{BB2003} is the existence of solution with further $W^{1,d}$-regularity, similar to the Helmholtz solution $\bu_{\text{Hel}}$ that cannot be ensured
to be uniformly bounded since it depends linearly on $f$.

\subsection{Existence of solutions for $L^{d,\infty}$ data.}

One first observation is that the Bourgain-Brezis problem 
has also a positive answer for the larger Lorentz space $Y=L^{d,\infty}(\Omd)$. Recall that a measurable function $f$ is in $L^{d,\infty}(\Omd)$ if and only if
$|\{x\in\Omd: |f(x)|> t\}|\le C^dt^{-d}$, $t>0$, and the smallest $C$ for which this holds is its $L^{d,\infty}(\Omd)$ norm.

\begin{theorem} 
There exists a constant $\gamma=\gamma_d$ such that for any $f\in \Leb^{d,\infty}(\Om)$, there exists  $\bu\in L^\infty(\Omega)$ satisfying
\begin{equation}\label{eq:divweak}
\div \bu=f, \quad   \|\bu\|_{L^\infty} \leq \gamma \|f\|_{\Leb^{d,\infty}(\Om)}.
\end{equation}
\end{theorem}

\begin{proof}
By definition of $Y=L^{d,\infty}(\Omd)$, one has
for any $f\in Y$ and measurable $E$,
$$
\int_E |f|=\int_{t>0} |\{x\in E\; : \; |f(x)|>t\}| \dt
\leq \int_{t>0}\min\{|E|, \|f\|_{L^{d,\infty}}^d t^{-d}\} \dt
\leq \frac {d}{d-1}\|f\|_{L^{d,\infty}} |E|^{\frac {d-1}d}.
$$
Therefore \iref{eq:aprioriper} follows by application of the isoperimetric inequality
with $\gamma_d=\frac {d}{d-1}K_d$ where $K_d$ is the isoperimetry constant \cite[3.2.43]{Fed1969}. 
\end{proof}

\begin{remark}
\label{remlor}
An equivalent proof consists in establishing that $BV(\RR^d)$ has continuous embedding
in the dual Lorentz space $Y^*=L^{d',1}(\RR^d)$, that is
\begin{equation}\label{eq:GNS}
\|g\|_{\Leb^{d',1}(\RR^d)} \leq  \beta_d|g|_{TV}.
\end{equation}
This readily follows by using the expression of the 
$L^{d,1}(\RR^d)$
norm through the distribution function
$$
\|g\|_{\Leb^{d',1}(\RR^d)}  =d'\!\int \limits_0^\infty |\{x\in \Om\; : \; |g(x)|>t\}|^{\frac{d-1}{d}}\dt,
$$
see e.g. \cite{BC2011}, and invoking the co-area formula and isoperimetric inequality to bound this
quantity by the total variation of $g$.
\end{remark}

\begin{remark}
\label{remweak}
The property assering that
$$
\int_E |f| \leq C|E|^{\frac {d-1}d}
$$
holds
for all measurable sets $E$ is actually equivalent to 
the membership of $f$ in $L^{d,\infty}(\Omd)$.  The smallest $C$ for which this is valid gives an equivalent norm for $L^{d,\infty}$.  We shall use this norm in going forward in this paper.
\end{remark}

\subsection{Beyond $L^{d,\infty}$}

What is the largest Banach space $Y$ of Borel measures $\mu$ which can be expressed 
as divergences of uniformly bounded $\bu$? Placing ourselves in
$\Omega=\RR^d$, we know that
such a $Y$ should be embedded in 
$G(\RR^d)$ the dual of $W^{1,1}_{\rm hom}(\RR^d)$,  
that is, for all $\mu\in Y$ and $g\in W^{1,1}_{\rm hom}(\RR^d)$ one has
\begin{equation}
\int_{\RR^d} g \dmu \leq \gamma \|\dmu\|_Y|g|_{TV}.
\label{eq:apriorimeas}
\end{equation}
Using the co-area formula, this is ensured in particular if 
\begin{equation}
|\mu(E)|\leq \gamma\|\mu\|_Y {\rm per}(E)
\label{eq:aprioripermeas}
\end{equation}
for all sets $E\subset \RR^d$ of finite perimeter.

Let us introduce the linear space of measures $S^d(\RR^d)$ that satisfy the condition
\begin{equation}
|\mu|(B)\leq C R^{d-1},  \quad R>0,
\label{eq:guydavid}
\end{equation}
for all balls $B$ or radius $R$, equipped with the norm
\begin{equation}
\|\mu\|_{S^d} :=\sup  R^{1-d} |\mu|(B) ,
\label{eq:normguydavid}
\end{equation}
where the supremum is taken over all balls $B$. 
For a general domain $\Omega$, we define 
$S^d(\Omega)$ in a similar manner, replacing $B$ by $B\cap \Omega$, and observe that any measure
in this space has its extension by $0$ contained
in $S^d(\RR^d)$ with a smaller or equal norm.

In dimension $d=2$, for positive measures, 
the condition $\mu(B)\leq CR$ was introduced and studied by Guy David
for Dirac measures on a curve $\Gamma$. He proved that this condition
is equivalent to the Ahlfors regularity of $\Gamma$ and to the boundedness
of the Cauchy integral operator acting on $L^2(\Gamma,\mu)$.

A distinction should be made between $S^d(\Omega)$ and the Morrey space $M^d(\Omega)$
that consists of all locally integrable
$f$ such that, for all ball $B$ of radius $R$,
\begin{equation}
\int_{B\cap \Omega} |f| \leq C R^{d-1},
\label{eq:morrey}
\end{equation}
with norm defined in a similar manner. This space is
included in $S^d(\Omega)$ with equal norm when $\mu$ is of the form $f\, \dx$, but the inclusion is strict: consider for example $\mu$ to be the Dirac
measure on a segment of the plane in the  case $d=2$.
In view of Remark \ref{remweak}, we have
$$
L^{d,\infty} \subset M^d \subset S^d,
$$
and these inclusions are strict. The following result
that follows the arguments from \cite{Mey2002}
and \cite{PT2017}, shows that the Bourgain-Brezis problem 
has also a positive answer for 
$Y=M^d(\Omd)$ and $Y=S^d(\Omega)$.

\begin{theorem} 
There exists a constant $\gamma=\gamma_d$ such that the following holds.
For any $f\in M^d(\Omega)$, there exists $\bu\in L^\infty(\Omega)$ satisfying
\begin{equation}\label{eq:divmorrey}
\div \bu=f, \quad   \|\bu\|_{L^\infty} \leq \gamma \|f\|_{M^d}.
\end{equation}
For any $\mu\in S^d(\Omega)$, there exists  $\bu\in L^\infty(\Omega)$ satisfying
\begin{equation}\label{eq:divmeas}
\div \bu=\mu, \quad   \|\bu\|_{L^\infty} \leq \gamma \|\mu\|_{S^d}.
\end{equation}
\end{theorem}

\begin{proof}
Without loss of generality we work on $\Omega=\RR^d$.
As a first step, we use the boxing inequality \cite[Theorem 2.11]{PT2008} 
that states that any open set $E\subset \RR^d$ of finite perimeter can be
covered by balls $B_j$ of radius $R_j$ such that
\begin{equation}
\sum_j R_j^{d-1} \leq C{\rm per}(E),
\end{equation} 
where the constant $C$ depends only on $d$. This shows that for any $f\in M^d(\RR^d)$, we have
$$
\int_E |f| \dx\leq C \|f\|_{M^d} {\rm per}(E).
$$
which implies \iref{eq:aprioriper} and therefore proving 
\iref{eq:divmorrey}. 

Similarly, for any $\mu\in S^d(\RR^d)$, we have
$$
|\mu|(E)\leq C \|\mu\|_{S^d} {\rm per}(E).
$$
For any test function $g\in \cD(\RR^d)$, we write $g=g_+-g_-$ and
$$
\Big |\int_{\RR^d} g \D\mu \Big| \leq \int_{\RR^d} g_+\D|\mu|+\int_{\RR^d} g_-\D|\mu|.
$$
For the first term, we have
$$
\int_{\RR^d} g_+\D|\mu|=\int_{0}^\infty |\mu|(E_t) \dt \leq C \|\mu\|_{S^d} \int_{0}^\infty
{\rm per}(E_t)
$$
where $E_t:=\{x\; :\; g(x)> t\}$. With a similar treatment of the second term
and using the co-area formula, we reach 
$$
\Big |\int_{\RR^d} g \D\mu \Big| \leq C \|\mu\|_{S^d}|g|_{TV}
$$
which  shows that $\mu$ belongs to the space $G(\RR^d)$ 
and thereby proves \iref{eq:divmeas}.
\end{proof}

\begin{remark}
\label{remarkdual}
We stress that, in contrast to the functions of $M^d(\RR^d)$, the measures of $S^d(\RR^d)$ belong to $G(\RR^d)$ but not to the dual of $BV(\RR^d)$. This is due to the fact that the trace of a 
$BV$ function on a $d-1$ dimensional surface could be meaningless. For example
if $\mu$ is the Dirac measure on a segment of the $2d$ plane, we cannot apply it
to the $BV$ function $g=\chi_Q$ where $Q$ is a square that admits this segment
as one of its side.
\end{remark}

\begin{remark}
As pointed out in \cite{Mey2002} in the case $d=2$,
a positive measure belongs to $G(\RR^d)$ if and only if
it belongs to $S^d(\RR^d)$. Indeed, on the one hand the above result shows 
that $S^d(\RR^d)$ is contained in $G(\RR^d)$. On the other hand, if $\mu$ is a positive
measure that is contained in $G(\RR^d)$ then to any ball $B=B(x_0,R)$ we associate the
$W^{1,1}$ function
$$
g(x)=\max \left\{0,2-\frac {|x-x_0|}{R} \right\}
$$
Since $\mu$ is positive and belongs to $G(\RR^d)$, and since $g$ is positive and larger than $1$ on $B$, we find that
$$
\mu(B) \leq  \int g \D\mu \leq C|g|_{TV}.
$$
On the other hand, it is easily checked that $|g|_{TV}\leq C_d R^{d-1}$ 
where $C_d$ only depend of $d$, therefore proving that $\mu\in S^d(\RR^d)$.
\end{remark}

\section{Explicit constructions of bounded solutions}\label{sec:bounded}

\subsection{A one-step explicit construction for $L^2$-data}

What follows is probably the simplest and most instructive construction of bounded
solutions to the Bourgain-Brezis problem \eqref{diveq}, at least in the  $d=$-case
with $Y=L^2(\Omd)$. Again, without loss of generality we will work on $\Omega=\RR^2$.

For any $(x,y)\in \RR^2$ and any fixed $f\in L^2(\RR^2)$, we define
\be
\label{HV}
V^2(x) :=\int_{-\infty}^{+\infty} |f(x,y)|^2  \, \dy,\quad     H^2(y):=\int_{-\infty}^{+\infty} |f(x,y)|^2\, \dx
\ee
and
\be
\alpha(x,y):=\frac{V(x)}{H(y)+V(x)},\quad  \beta(x,y):= \frac{H(y)}{H(y)+V(x)}.
\ee
We then consider the splitting $f=f_1+f_2$ where
\be
\label{fone}
f_1(x,y):=\alpha(x,y)f(x,y),\quad{\rm and} \quad f_2(x,y):=\beta(x,y)f(x,y),
\ee
and we define 
\be
\label{define1}
u_1(x,y):=\int_{-\infty}^x f_1(s,y)\dds \quad {\rm and}\quad u_2(x,y):=\int_{-\infty}^y   f_2(x,t)\dt
\ee
Therefore $\bu=(u_1,u_2)$ satisfies $\div \bu=f_1+f_2=f$ and it remains to check  
that $u$ is uniformly bounded.  Let us bound $|u_1(x,y)|$ for any arbitrary but fixed $(x,y)\in\RR^2$.  If $H(y)=0$, then obviously $u_1(x,y)=0$ for all $x$ so we assume $H(y)\neq 0$.  Then, for any $x$ we have
\begin{eqnarray*}
\label{normu}
|u_1(x,y)|&\le& \int_{-\infty}^\infty |f_1(s,y)|\, \dds\le \int_{-\infty}^\infty |f(s,y)| \frac{V(s)}{H(y)}\, \dds\\
&\le&  H(y)^{-1}\left(\int_{-\infty}^\infty|f(s,y)|^2\, \dds\right)^{1/2}\left(\int_{-\infty}^\infty V(s)^2\, \dds\right)^{1/2}\\
&=&  \left(\int_{-\infty}^\infty V(s)^2\, \dds\right)^{1/2}=\|f\|_{L^2(\RR^2)}.
\end{eqnarray*}
In a similar way, we obtain the bound $\|u_2\|_{L^\infty} \leq \|f\|_{L^2(\RR^2)}$, and so
$\bu$ is a bounded solution to the Bourgain-Brezis problem for $f$.

\begin{remark}
The above
splitting of $f$ into $f_1$ and $f_2$ is designed to ensure that the univariate primitive $u_1$
and $u_2$ 
are uniformly bounded. An interesting variant consists in taking 
\be
f_1(x,y):=f(x,y) \chi_{\{H(y)\leq V(x)\}} \quad {\rm and} \quad 
f_2(x,y):=f(x,y) \chi_{\{V(x)< H(y)\}}
\ee
for which it is easily checked that the solution $u$ also has each
component $u_1$ and $u_2$ uniformly bounded by $\|f\|_{L^2}$. The extra feature
of this choice is that $f_1$ and $f_2$ have disjoint supports. As we discuss next, in the more general $d$-dimensional case 
with $Y=L^d(\RR^d)$, it is 
possible to explicitly construct a splitting $f=f_1+\cdots+f_d$ also with disjoint supports and
such that the univariate primitive $u_j$ of $f_j$ with respect to $x_j$ are uniformly
bounded.
\end{remark}

\subsection{A $(d-1)$-step explicit construction for $L^{d}$-data}
We now consider the general $d$-dimensional case
with data $f\in L^d(\RR^d)$ . Given any such $f$, we construct 
 $d$ pairwise disjoint sets  
   $\Omega_j=\Omega_j(f)$ with $\RR^d=\bigcup_{j=1}^d\Omega_j$, so that  the functions
 $f_j:=f\chi_{\Omega_j}$ satisfy $f=f_1+\cdots +f_d$ as well as 
 \be
 \label{Td}
 \int_{\RR} |f_j(x_1,\dots, x_{j-1},s,x_{j+1},\dots, x_d)|\,\dds
 \le \|f\|_{L^d(\RR^d)},
 \quad j=1,\dots,d.
 \ee
 In turn the functions
\be
u_j(x_1,\dots,x_d):=\int_{-\infty}^{x_j} f_j(x_1,\dots, x_{j-1},s,x_{j+1},\dots, x_d)\,\dds,\quad j=1,\dots,d,
\ee
satisfy $\|u_j\|_{L^\infty(\RR^d)}\le \|f\|_{L^d(\RR^d)}$.  Hence, $\bu=(u_1,\dots,u_d)$ is a solution to \iref{diveq} with $\|\bu\|_{L^\infty(\RR^d)}\le \|f\|_{L^d(\RR^d)}$.  
 
 The construction proceeds  by induction on $d$.  When $d=1$,  given any $f\in L^1(\RR)$, we define $\Omega_1=\RR$ in which case the above claim is obvious.  We assume we have shown how to construct such sets $\Omega_1(g).\dots,\Omega_{d-1}(g)$ whenever $g\in L^{d-1}(\RR^{d-1})$ and give the construction of $\Omega_1(f),\dots,\Omega_d(f)$ whenever $f\in L^d(\RR^d)$. Without loss of generality, we can assume $\|f\|_{L^d(\RR^d)}=1$.   
   
 We write any vector $x\in\RR^d$ as $(x_1,y)$ where $y=(x_2,\dots,x_d)\in\RR^{d-1}$ and define the thresholds
$$
t(y)^{d-1}:= \int_\RR |f(x_1,y)|^d {\dx}_1,\quad y\in \RR^{d-1}.
$$
Let 
\be
\label{defO1}
\Omega_1:=\Omega_1(f):=\{x=(x_1,y)\in\RR^d: |f(x_1,y)|\ge \tau(y)\},
\ee
and $\Omega_1'$  be its complement in $\RR^d$. We define 
\be
\label{deffg}
f_1:=f\chi_{\Omega_1} \quad {\rm and }\quad g:=f-f_1=f\chi_{\Omega_1'}.
\ee
This determines $u_1$ and for any $x=(x_1,y)\in \RR^d$,  we have
$$
 \int_\RR |f_1(s,y)| \dds \leq
\int_\RR |f(s,y)| \left( \frac{|f(s,y)|}{t(y)} \right)^{d-1}\dds
= t(y)^{1-d} \int_\RR |f(s,y)|^d \dds=1.
$$ 
This shows that \iref{Td} holds for $f_1$ and $\|u_1\|_{L^\infty(\RR)}\le 1$ as desired.

We proceed to construct the set $\Omega_2,\dots,\Omega_d$.  For any fixed $x_1 \in\RR$, we consider the function $g(x_1,y)$ as a function of $y\in\RR^{d-1}$.  We have
\be
\label{checkg}
\int_{\RR^{d-1}}  |g(x_1,y)|^{d-1} \dy
\leq \int_{\RR^{d-1}}t(y)^{d-1} \dy=\|f\|_{L^d(\RR^d)}^{d} =1.
\ee
From the induction hypothesis, 
we can apply our construction to $g(x_1,\cdot)$
which is a function of $d-1$ variables $y=(y_2,\dots,y_d)$.
This gives $d-1$ disjoint sets $\Omega_j(x_1)$ for $j=2,\dots,d$ whose union is
$\RR^{d-1}$ and for which the above results are valid.
Therefore $g(x_1,\cdot)$.
is split into functions $g_j(x_1,\cdot)=g(x_1,\cdot)\chi_{\Omega_j(x_1)}$
which satisfy
\be
\int_{\RR } |g_j(x_1,y_2,\dots,y_{j-1}, s ,y_{j+1},\dots,y_{d}) |\,\dds \le 1,\quad j=2,\dots,d.
\ee
We now define
$$
\Omega_j=\Omega_j(f):= \{(x_1,y)):\ x_1\in\RR, \  y\in \Omega_j(x_1)\},\quad j= 2,\dots,d.
$$
This completes the definition of the sets $\Omega_j$ and the functions $f_j$ and $u_j$.  We are left to check \eref{Td} for $j\neq 1$. Since 
$f_j(x_1,y)=  g_j(x_1,y)$, it suffices to write 
 $$ 
 \int_{\RR} |f_j(x_1,\dots, x_{j-1},s,x_{j+1},\dots, x_d)|\,\dds=\int_{\RR } |g_j(x_1,y_2,\dots,y_{j-1}, s ,y_{j+1},\dots,y_{d-1}) |\,\dds \le 1,
$$
This establishes the properties we want of our construction for $d$ dimensions.

\subsection {A constructive decomposition for  $L^{2,\infty}$}  
We know from the theoretical results of \S 2, that the space 
$Y=L^{d,\infty}(\Omega)$ is admissible 
whenever $\Omega\subset\RR^d$  is measurable.  In this section, for any $\tau>1$
and any bounded measurable set $\Omega\subset \RR^2$, we  give an algorithm that takes any $f\in L^{2,\infty}(\Omega)$ and  constructs a solution $\bu$ to \iref{diveq} such that $\|\bu\|_{L^\infty} \le \tau \|f\|_{L^{2,\infty}}$.

 We fix $\Omega$ and $f$ in going forward.   We  assume without loss of generality that  
 $\Omega=[-R,R]^2$, for some $R>0$ and $\|f\|_{L^{2,\infty}(\Omega)} =1$. 
We define $f$ to be zero outside of $\Omega$.       
We  construct a disjoint splitting $\Omega=\Omega_1\cup \Omega_2$ and $f_j:=f\chi_{\Omega_j}$, $j=1,2$, and take
 \be
\label{defu}
u_1(x,y):=\int_{-\infty}^xf_1(s,y)\, \dds, \quad u_2(x,y):=\int_{-\infty}^yf_2(x,s)\, \dds.
\ee
Thus $\div(\bu)=f$ when $\bu:=(u_1,u_2)$, and the only issue will be to show that
  \be
\label{toshow}
 \int_{-\infty}^\infty |f_1(s,y)|\, \dds\le \tau , \quad \int_{-\infty}^\infty |f_2(x,s)|\, \dds\le \tau,
\ee
  for all $x,y\in\RR$.

   Let us first note that $f\in L^1(\RR^2)$ and 
   \be
   \label{L_1norm}
   M:=\|f\|_{L^1(\RR^2)} =\int_{\Omega} |f|\le |\Omega|^{1/2},
   \ee
   where we used Remark \ref{remweak}.
   We define the horizontal line $L_H(y):=\{(x,y): \ x\in [-R,R]\}$ at level $y$ and the vertical line $L_V(x):=\{(x,y): \ y\in[-R,R]\}$ at level $x$.  For  any measurable  function $g$ that is   
 supported on $\Omega$, we  define the energies
\be 
\label{energies}
E_H(g,y):=\int_{L_H(y)} |g(x,y)|\, \dx, \ y\in [-R,R], \quad  E_V(g,x):=\int_{L_V(x)} |g(x,y)|\, \dy, \ x\in [-R,R],
\ee
which may be infinite.

Here is the first step of our construction.  Let $A:=\{y\in [-R,R]:\  E_H(f,y)\le \tau\}$ and $A':=\{y\in [-R,R]:\  E_H(f,y)> \tau \}$
and 
 $$
 \Omega_H:=\{(x,y)\in\RR^2: E_H(f,y)\le \tau \}= \bigcup_{y\in A} L_H(y).
 $$
We define 
$f_1:=f $  on $\Omega_H$ and $f_2:= 0$ on $\Omega_H$.     Notice that  we have
\be
\label{defu1}
 \int_{-\infty}^x |f_1(s,y)|\, \dds\le \tau,\quad y\in A.
\ee
This means that \eref{toshow} is satisfied for $y\in A$. 

We proceed to the second step of our construction. 
Let $B:=\{x\in [-R,R]:\  E_V(f,x)\le\tau \}$ and $B':=\{x\in [-R,R]:\  E_V(f,x)> \tau \}$
and 
 $$
 \Omega_V:=\{(x,y)\in\Omega: E_V(f,x)\le \tau \}= \bigcup_{x\in B} L_V(x).
 $$
We define $f_2=f$ on $\Omega_V\setminus \Omega_H$ and $f_1$ is defined to be zero on this set.  We have
\be
\label{defu2}
 \int_{-\infty}^y |f_2(x,s)|\, \dds\le \tau,\quad x\in B.
\ee
 
Thus far, we have defined $f_1$ and $f_2$ on $\Omega_H \cup\Omega_V$.   Let  $\Omega':= \Omega\setminus(\Omega_H\cup\Omega_V)$.  The important thing to notice is
that $\Omega'$ is gotten from $\Omega$ by removing horizontal and vertical strips.  The following lemma shows that we have removed a significant portion of $\Omega$ in this construction.

\begin{lemma}
\label{L:C1}
The measure of $\Omega'$ satisfies
\be
\label{LC1}
|\Omega'|\le \tau^{-2}|\Omega|.
\ee
\end{lemma}
\noindent
{\bf Proof:}  To prove this claim, we observe that
  $$
  \Omega':=\{(x,y)\in\Omega: E_H(f,y)\ge \tau\  {\rm and} \ E_V(f,x)\ge \tau\}=A'\times B'.
  $$
Let $a:=|A'|$ and $b:=|B'|$ be the univariate Lebesgue measure of these sets. Then, we have
\be 
a\tau \le\ \int_{\Omega}|f(x,y)\,\dx\dy\le |\Omega|^{1/2}. 
\ee 
A similar argument gives $b\tau\le |\Omega|^{1/2}
$.  Hence, we have 
\be
\label{have1}
|\Omega'|=ab\le \tau^{-2}|\Omega|,
\ee
which proves the lemma.\hfill $\Box$

After applying the first step of our construction, we have defined $f_1$ and $f_2$ outside of $\Omega'$.   Let $\Omega^1:=\Omega'$ and let us repeat
our  construction for the set $\Omega^1$ in place of $\Omega$.  This gives a new set $\Omega^2:=[\Omega^1]'\subset\Omega^1$  and thereby give the definitions of $f_1,f_2$  outside of $\Omega^2$. 
The new residual set $\Omega^2$ satisfies $|\Omega^2|\le \tau^{-2}|\Omega^1|\le \tau^{-4}|\Omega|$.  Iterating this procedure gives in the limit a definition of $f_1$ and $f_2$ on all of $\Omega$ except for a set of measure zero.   One easily checks that \eref{toshow} holds.  For example to check this for $f_1$, we note that if $f_1$ is defined to be nonzero on
on a line $L_H(y)$ then at the (first) step $k$ where it is defined to be nonzero it is completely defined on this line and \eref{toshow} holds on this line.

We leave as an open problem the construction of decompositions with the above properties for the general case $f\in L^{d,\infty}(\RR^d)$ with $d\ge 2$.

 \section{Variational-based constructions}\label{sec:variat}
 
 \subsection{Minimization problems}\label{sec:min}

 One natural way of approaching bounded solutions to \iref{diveq}
 for data $f\in Y(\Omega)$ is to consider the minimization
 of functionals of the form
 \be
 \Fvee(\bu)=\Fvee_{\lambda}(\bu)=\|\bu\|_{L^\infty}+\lambda \|f-\div \bu\|_Y^p 
 \label{genfunc}
 \ee
 for some fixed $p\geq 1$. Indeed, one intuition is that if a uniformly bounded solution to \iref{diveq} exists,
 the minimizer of $\Fvee_{\lambda}$  should tend to the solution $\bu$ of \iref{diveq}
 with minimal $L^\infty$ norm as $\lambda\to \infty$.
 
 We shall first see that in the case of $Y=L^2$ and $p=2$, it is possible to avoid
 letting $\lambda\to \infty$ through a two-step constructive approach. We then
 discuss more general situations where we can construct a uniformly bounded
 solution $\bu$ by a hierarchical decompositions based on iterated
 minimizations of the above functional.
 
Before going further let us observe that  
the existence of a minimizer for $\Fvee$ can be derived by the 
elementary arguments under a mild assumption on the space $Y$.

\begin{lemma}\label{lem:min-exists} 
Assume that $Y=Z'$ is a dual space of distribution, so that 
$$
\|v\|_Y:=\max\{\langle v,\varphi \rangle_{Y,Z}\; :\; \varphi \in \cD(\Omega), \; \|\varphi\|_Z\leq 1\},
$$
then there exists a minimizer $\bu_1$ of $\Fvee$
\end{lemma}

\begin{proof}
Consider a minimizing sequence $\bu^n$, therefore such that $\|\bu^n\|_{L^\infty}$ and $\|\div \bu^n\|_{Y}$ are
are uniformly bounded. Then, up to a subsequence extraction, we have the following properties :
\begin{itemize}
\item[(i)] Both $\|f-\div \bu^n\|_{L^2}$ and $\|\bu^n\|_{L^\infty}$ have limits $A$ and $B$ such that
$A+\lambda B$ is the infimum of $\Fvee$.
\item[(ii)] $\bu^n$ converges in the $L^\infty$ weak-$*$ sense to some $\bu_1\in L^\infty$.
\item[(iii)] $\div \bu^n$ converges in the $Y$ weak-$*$ sense to the (weak) divergence $\div \bu_1\in L^2$.
\end{itemize}
From this and the properties of weak lower semi-continuity of norms, it readily follows that $\bu_1$ is a minimizer of $\Fvee$.
\end{proof}

Existence is therefore ensured for reflexive spaces $Y$ such as $L^d(\Omega)$ in the $d$-dimensional case,
but also for $Y=L^{d,\infty}$ which we have seen earlier to be an admissible choice for the existence of
uniformly bounded property. We stress that uniqueness of minimizers, is in general not ensured, however
in the case where $Y$ is strictly convex, such as $L^d$, we find that $\div(\bu_1)$ is unique.

Note that the minimization of $\Fvee$ may be computationally intensive,
depending on the form of the $Y$ norm. In 
the particular case of $Y=L^d$,
it can be computed by solving relatively simple 
Euler-Lagrange equations, see \cite{TT2011}.

\subsection{A two-step approach for $L^{2}$-data}\label{sec:2-step}

Consider the case of a domain $\Omega\subset \RR^2$ and $Y=L^2(\Omega)$. 
With the choise $p=2$, the functional of interest is therefore
\be
\Fvee(\bu)=\|\bu\|_{L^\infty}+\lambda \|f-\div \bu\|_{L^2}^2
\label{L2func}
\ee
An interesting property of the minimizers is given by the following.

\begin{lemma}
\label{lem:res}
Fix $\lambda>0$ and let $r_\lambda=f-\div\bu_\lambda$ be the residual of the equation \eqref{diveq} for a
minimizer $\bu_\lambda$ of \iref{L2func}. Then $r_\lambda$ belongs to $BV(\Omega)$
with
\begin{equation}
\label{bv}  
|r_\lam|_{TV} \leq \frac{1}{2\lam}.
\end{equation} 
\end{lemma}

\begin{proof}
For any $\z\in \cD(\Omega)$ and $\epsilon >0$, we have
\begin{eqnarray*}
\Fvee(\bu_\lambda)= \Bnorm{\u_\lambda} +\lambda\|f-\div\u_\lam\|^2_{L^2} & \leq &  \Bnorm{\u_\lambda+\epsilon\z} +\lambda\|f-\div(\u_\lam+\epsilon \z)\|^2_{L^2}   \\
&  \leq   & \Bnorm{\u_\lam}+\epsilon  \Bnorm{\z} + \lambda 
\displaystyle \|r_\lambda\|^\rp_{L^2} -2 \lam\epsilon \int_\Omega r_\lambda \div \z+o(\epsilon)
. \\
& = & \Fvee(\bu_\lambda) +\epsilon  \Bnorm{\z}  -2 \lam\epsilon \int_\Omega r_\lambda \div \z +o(\epsilon)
\end{eqnarray*}
and by letting $\epsilon \downarrow 0$ we find that
\be
\int_\Omega r_\lambda \div \z \leq \frac 1{2\lam} \Bnorm{\z},
\ee
for all $\z\in \cD(\Omega)$, which shows that $r_\lambda\in BV(\Omega)$ with bound \iref{bv} for its total variation.
\end{proof}

Note that we also have the trivial bounds 
\be
\label{triv}
\|r_\lambda\|_{L^2}\leq \|f\|_{L^2} \quad{\rm and} \quad \|\bu\|_{L^\infty} \leq \lambda\|f\|_{L^2}^2,
\ee
by comparing $\Fvee(\bu_\lambda)$ with $\Fvee(0)$. However Lemma \ref{lem:res}
shows a ``regularization effect'' $f\in L^2 \mapsto \r_\lambda\in BV$.
As noted in Remark \ref{remlor}, the space $BV$ 
has a continuous embedding in the Lorentz space $L^{2,1}$ which is strictly smaller than $L^2$. 

This effect leads us to a direct construction of a bounded solution. Without loss of generality, we again 
work on $\Omega=\RR^2$, and denote by $\bu_1$ the minimizer 
and $r_1=f-\div \bu_1$ the residual, when taking the particular value $\lambda:=\|f\|_{L^2}^{-1}$.
Using both \iref{bv} and \iref{eq:GNS}, we have on the one hand
\be
\|r_1\|_{L^{2,1}}\leq \beta_2 |r_1|_{TV}
\leq \frac {\beta_2}{2\lambda} =\frac {\beta_2}{2}\|f\|_{L^2};
\ee
and on the other hand
\be
\|\u_1\|_{L^\infty} \leq \|f\|_{L^2},
\ee
in view of the second bound in \iref{triv}. We then write $r_1 =\div \u_2$, where 
\be
\u_2:= \psi* \r_1= \frac{1}{2\pi}\frac{x}{|x|^2}*r_1,
\ee
is the Helmoltz solution for the data $\r_1$. Since $\psi\in L^{2,\infty}$ and $\r_1\in L^{2,1}$, it is readily seen
that $\u_2$ is uniformly bounded by
\be
\|\bu_2\|_{L^\infty}\leq \|\psi\|_{L^{2,\infty}}\|r_1\|_{L^{2,1}} \leq  C\|f\|_{L^2}, \quad C:=\frac{\beta_2}{2}\|\psi\|_{L^{2,\infty}}.
\ee
Thus we end up with
\be
\u_{\text{2step}}:=\u_1+\u_2,
\ee 
as a two-step construction of a uniformly bounded solution to \iref{diveq}
which satisfies \iref{eq:divU=F} with $\gamma=1+C$.

\begin{remark}
A similar regularization effect takes place in the $d>2$ case for data $f\in L^d$
\[
f\in L^d(\Omega) \mapsto r_1 \in L^{d,d-1}(\Omega).
\]
However, since $\psi\in L^{d,\infty}$, this is not enough to derive a similar two-step construction 
by applying the Helmoltz solver to the residual. Instead, this will be addressed by the multi-step construction in the next section below. 
\end{remark}

Figure \ref{fig:2step}, quoted from \cite[Section 6]{TT2011},  shows the two-step
solution of the example due to L. Nirenberg, \cite[Remark 7]{BB2003}, which demonstrates the unboundedness of $\|\u_{\text{Hel}}\|_{L^\infty}$ solved for $\u\in L^2_{\#}([-1,1]^2)$ with periodic boundary conditions, given by
\begin{equation}\label{eq:counter}
f=\Delta v \qquad   v(x_1,x_2) := x_1|\log |x||^{1/3}\zeta(|x|), \qquad 
 \zeta(r)=\chi_{(-1,1)}e^{-\frac{1}{1-r^2}}.
\end{equation}
 In this case, 
 Helmholtz solution, $\u_{\text{Hel}}=\nabla V$, has a fractional logarithmic growth at the origin, which should be contrasted with the hounded two-step constructed solution shown in figure \ref{fig:2step}.
 Table \ref{tab:mesh2} reports that the ratio between $N\times N$ grid discretization of 
$\|\u^{N}_{\text{2step}}\|_{L^\infty}$ and $\|f^N\|_{L^2}$ remains bounded
when $N$ is large, in contrast to the computed solution of Helmholtz.

\begin{figure}[htb]
\centering
\includegraphics[scale=.57]{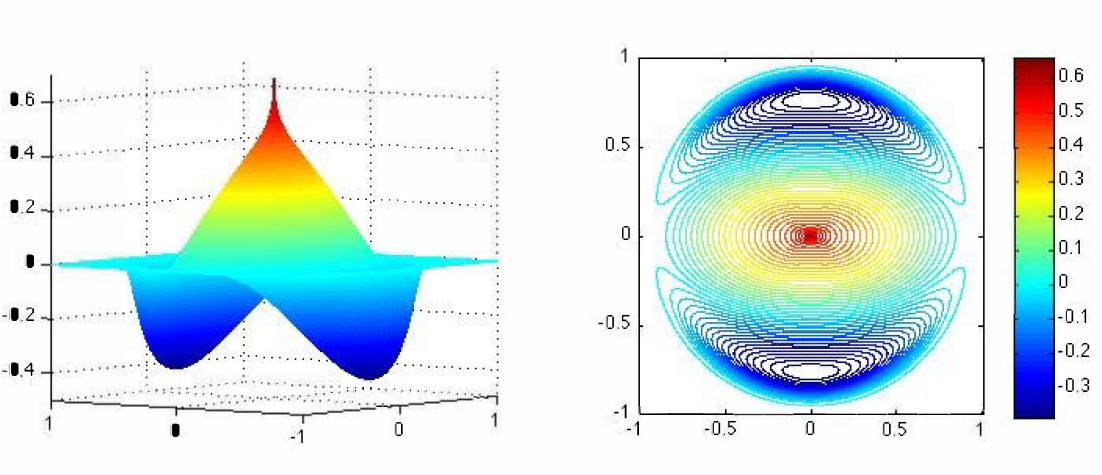}
\includegraphics[scale=.57]{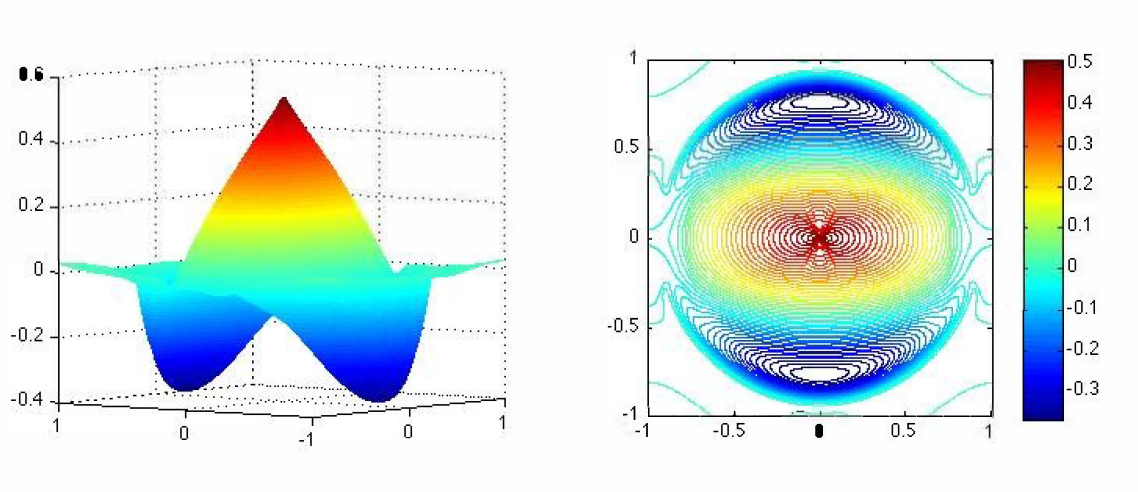}
\caption{Solution of Bourgain-Brezis problem with 2D data in \eqref{eq:counter}.\newline \hspace*{2.1cm}Helmholtz solution, 
$\u_{\text{Hel}}$ (top), vs. two-step solution, $\u_{\text{2step}}$ (bottom).}
\label{fig:2step}       
\end{figure}

\begin{table}[ht]
\renewcommand{\arraystretch}{1.4}
\begin{tabular}{cccccc}
\hline\hline
The  $N\times N$ grid &$50\times 50$&$100\times 100$&$200\times 200$&
$400\times 400$&$800\times 800$\\[0.3cm]
$\displaystyle \frac{\|\u^{1,N}_{\text{Hel}}\|_{L^\infty}}{\|f^N\|_{L^2}}$&0.2295&0.2422&0.2540&0.2650&0.2752\\[0.5cm]
$\displaystyle \frac{\|\u^{1,N}_{\text{2step}}\|_{L^\infty}}{\|f^N\|_{L^2}}$&0.2096&0.2128&0.2144&0.2151&0.2154\\[0.4cm]
\hline\hline
\end{tabular}

\medskip
\caption{$\Leb^{\infty}$ norm of numerical solutions for different
 grids: Helmholtz vs. the two-step solution of \eqref{eq:counter} for different
 grids.}\label{tab:mesh2}
\end{table}

\subsection{Hierarchical constructions for data in Fr\'echet differentiable spaces.}

We now work with general data $f\in Y(\RR^d)$.
In this section, we use a hierarchical approach
to construct uniformly bounded solution, under the 
assumption that the $Y$ norm is Frechet differentiable. 

\begin{assumption}[{\bf \Frechet differentiablity}] The $Y$-norm  is   \emph{\Frechet differentiable}, namely --- there exists $\phi: Y\to Y'$ such that
\begin{equation}\label{eq:Fre}
\|v+\epsilon w\|_{Y} =\|v\|_{Y} + 
\epsilon\langle \phi({v}),w\rangle +o(\epsilon) \ \ \textnormal{for all} \ \
v,w \in Y, \ \ v \neq 0. 
\end{equation}
\end{assumption}

As an immediate consequence of this assumption, for any $p>1$, the application
$v\mapsto \|v\|_Y^p$ is also Fr\'echet differentiable and its derivative
is given by
$$
\phi_p(v):=p\|v\|^{p-1}_Y\phi({v}).
$$
As we discuss further, spaces admitting Fr\'echet differentiable norms 
are for example $Y=L^d$ as well as $Y=L^{d,q}$ 
when $1<q<\infty$, but not $Y=L^{d,\infty}$.

Let us now consider for a $p>1$ the general functional $\Fvee_\lambda$
of \iref{genfunc}, assuming as before that $Y$ is a dual space. In the sequel we 
use the following two results which makes use of $\phi_p$.  
The first is the generalization of Lemma \ref{lem:res}
that shows a regularization effect, now
on $\phi_p(r_\lambda)$ where
$r_\lambda=f-\div\bu_\lambda$ is the residual.

\begin{lemma}
If  $\u_\lam\in \B$  is a minimizer of  \iref{genfunc}
with residual $\r_\lam=f-\div\u_\lam \in Y$ then
\begin{equation}
\label{condres}  
|\phi_p(\r_\lam)|_{TV} \leq \frac{1}{\lam}.
\end{equation} 
\label{lem:resgen}
\end{lemma}

\begin{proof}
For any test function $\z\in \cD(\Om)$, we have
\begin{eqnarray*}
 \ \ \JT_\lam(\u_\lam)= \Bnorm{\u_\lam} +\lambda\|f-\div\u_\lam\|^p_{Y}  & \leq & 
  \Bnorm{\u_\lam+\epsilon \z}  + \lambda \|f-\div(\u_\lam+\epsilon \z)\|^p_{Y} \\
&  \leq   & \Bnorm{\u_\lam}+|\epsilon|\cdot\Bnorm{\z} + \lambda 
\displaystyle \|r_\lam\|^p_{Y} -\lam\epsilon \big\langle \phi_p(r_\lam), \div \z \big\rangle +o(\epsilon)
. \\
& = & \JT_\lam(\u_\lam) + |\epsilon|\cdot\Bnorm{\z}-\lam\epsilon \big\langle \phi_p(r_\lam), \div \z \big\rangle+o(\epsilon),
\end{eqnarray*}
and by letting $\epsilon \downarrow 0$ we find
$\displaystyle 
|\phi_p(r_\lam)|_{TV}= \sup_{0 \neq \z \in \cD(\Omega)} 
\frac {\int_{\Om} \phi_p(r_\lam)\,\div \z\, }{\|\z\|_{L^\infty}}  \leq \frac{1}{\lam}$.
\end{proof}

\begin{remark}\label{rem:on-bu-lambda} Note that when $\lam  <\frac {1}{|\phi_p(f)|_{TV}}$,  we then have a trivial minimizer $\u_\lam =0$ and $r_\lam=f$. Lemma \ref{lem:resgen} is relevant
when $\lam$ is large enough
\begin{equation}\label{eq:admiss}
\lam > \frac{1} {|\phi_p(f)|_{TV}}.
\end{equation}
Then $\u_\lam \neq 0$, and \iref{condres} asserts  the $BV$ regularity 
of $\phi_p(r_\lam)$. In fact, for large enough $\lam$, one has
the equality $|\phi_p(r_\lam)|_{TV} = {1}/{\lam}$,  and the minimizer is characterized as an exteremal  pair in the sense that
$$
\int div \bu_\lambda r_\lambda =  |\bu_\lambda|_\infty |r_\lambda|_{TV},
$$
see \cite[Theorem 3]{Mey2002},\cite[Lemma A.3]{Tad2014}.
\end{remark}

\medskip
We also need a second priori estimate which will be useful as a \emph{closure bound} for the iterative procedure of  hierarchical construction described below.

\begin{lemma}
\label{lem:closure}
Assume that $Y$ has a Fr\'echet differentiable norm and that $BV$ is embedded in $Y'$ in the
sense that
\be
\|v\|_{Y'}\leq \beta |v|_{TV}, \quad v\in BV(\RR^d).
\label{bvYprime}
\ee
Then, the following  a priori bound holds
\begin{equation}\label{eq:closure}
\|v\|_{Y}^{p-1} \leq  \eta |\phi_p(v)|_{TV}, \quad\quad v\in Y,
\end{equation}
with $\eta=\beta/p$.
\end{lemma}

\begin{proof} Fixing $v\in Y$ and comparing the first order terms in 
\[
(1+\epsilon)^p\|v\|_{Y}^p    =\|v+\epsilon v\|_{Y}^p = \|v\|_Y^p + \epsilon\langle  \phi_p(v),v\rangle +o(\epsilon),
\]
implies $p\|v\|_{Y}^p = \langle  \phi_p(v),v\rangle$. Therefore
$$
p\|v\|_{Y}^p\leq \| \phi_p(v)\|_{Y'}\|v\|_{Y}\leq \beta |\phi_p(v)|_{TV}\|v\|_{Y},
$$
 which yields \iref{eq:closure}.
\end{proof}

Following \cite{Tad2014}, we fix the value $p=2$
and for a given sequence $(\lambda_j)_{j\geq 1}$, we define iteratively 
$\u_j$ as the minimizer of 
\be\label{eq:quadratic}
\Fvee_j(\bu)=\|\bu\|_{L^\infty}+\lambda_j \|r_{j-1}-\div \bu\|_Y^2,
\ee
and $r_j=r_{j-1}-\div \bu_j$. The following result shows that,
with a proper choice of $\lambda_j$,
the sum of the $\u_j$ admits a limit which is our desired uniformly bounded
solution to \iref{diveq}.

\noindent
\begin{theorem}\label{thm:var}
Consider a \Frechet differentiable space  $Y$ such
that \iref{bvYprime} holds and set $\lambda_j=\lambda_1 2^{j-1}$,
where $\lambda_1:=\frac{2\eta}{\|f\|_{Y}}$.
Then, for any given $f\in Y$, the 
sum of the $\u_j$ converges in $L^\infty$ to a limit
$\bu=\sum_{j=1}^\infty \u_j$ which is solution to \iref{diveq}
and satisfies
\begin{equation}\label{eq:UinLinf}
\|\bu\|_{\Leb^\infty} \leq 2\beta \|f\|_{Y}.
\end{equation} 
\end{theorem}

\begin{proof}
Comparing $\u_j$ with the $0$ solution, we find that
$$
\|\bu_j\|_{L^\infty} \leq \Fvee_j(0)=\lambda_j \|r_{j-1}\|^2_Y,
$$
and in particular
$$
\|\bu_1\|_{L^\infty} \leq \lambda_1 \|f\|^2_Y.
$$
On the other hand, combining the closure bound \iref{eq:closure}
and the regulatization bound \iref{condres}, we find that
\begin{equation}\label{eq:rj-decay}
\|r_{j-1}\|_Y\leq \eta |\phi_2(r_{j-1})|_{TV} \leq \frac {\eta}{\lambda_{j-1}}.
\end{equation}
Therefore, for $j\geq 2$
$$
\|\bu_j\|_{L^\infty} \leq \eta^2 \frac{\lambda_j}{\lambda_{j-1}^2}=\frac{8\eta^2}{\lambda_1} 2^{-j}
$$
and $\sum_{j=1}^\infty \u_j$ thus converges to a uniformly bounded limit with
$$
\|\bu\|_{L^\infty} \leq \lambda_1\|f\|^2_Y+\frac{4\eta^2}{\lambda_1}=4\eta \|f\|_Y=2\beta \|f\|_Y, 
$$
where we have used the chosen value of $\lambda_1$.
In addition, a telescoping sum of $r_j=r_{j-1}-\div \bu_j$ yields $f=\div (\sum_{k=1}^j\bu_k) + r_j$ and 
the residual $r_j$ tends to $0$ in $Y$.
This proves 
that $\div(\bu)=f$.
\end{proof}

\begin{remark}
Theorem \ref{thm:var} extends the hierarchical construction  of Bourgain-Brezis problem in \cite{Tad2014} for $Y=L^d$-data with $p=d$. In fact, the choice of the parameter $p>1$ need not be tied to $Y$, which led to the simpler choice of 
$p=2$ in \eqref{eq:quadratic}.  
\end{remark}

\begin{remark}
The closure \eqref{eq:closure} implies that our choice of $\lam_1$ is admissible in the sense that \eqref{eq:admiss} holds, 
\[
\lam_1=\frac{2\eta}{\|f\|_{Y}} \geq \frac{2\eta}{\eta|\phi_2(f)|_{TV}}
>\frac{1}{|\phi_2(f)|_{TV}}.
\]
In other words ---  already the first variational iteration produces a non-trivial minimizer, $\bu_1\neq 0$. In fact, one can underestimate $\lam_1 < \beta/\|f\|_{Y}$ in case $\beta$ in \eqref{bvYprime} is not accessible, and yet the variational iterations become effective after the first $\log(\beta)$  iterations with zero minimizers.
\end{remark}

\begin{example}
Inspired by \cite{Mey2002}, we demonstrate the hierarchical constriction of theorem \ref{thm:var} in the two-dimensional example  of  $f=\alpha \chi_R$, where $\alpha$ is a fixed constant and $\chi_R$ is the characteristic function of the ball of radius $R$. Of course, in this case of $BV$ function we can simply solve $\displaystyle \bu=\nabla\Delta^{-1}f= \alpha \frac{x}{2}\chi_R$. 
The minimization with $\lambda> 1/(4\pi \alpha)$ yields 
\[
f=div \bu_\lambda+r_\lambda, \quad \bu_\lambda= (\alpha-\beta)\nabla\Delta^{-1}\chi_R=(\alpha-\beta)\frac{x}{2}\chi_R \ \ \textnormal{and} \ \  r_\lambda=\beta\chi_R \ \ \textnormal{with} \ \ \beta:=\frac{1}{4\pi \lambda}.
\]
This is verified by checking that $|r_\lambda|_{TV}= \beta 2\pi R=\frac{1}{2\lambda}$,  and the extremal property (see remark \ref{rem:on-bu-lambda}),
\[
\int div \bu_\lambda r_\lambda = (\alpha-\beta)\frac{R}{4\lambda}= |\bu_\lambda|_\infty |r_\lambda|_{TV}, \quad |\bu_\lambda|_\infty=(\alpha-\beta)\frac{R}{2}.
\]
Iterating we find (for $\alpha>8\pi$)
\[
f= \sum_{j=1}^\infty \bu_j, \qquad \bu_j=\left\{\begin{array}{cc}
\displaystyle \Big(\alpha-\frac{1}{8\pi}\Big)\frac{x}{2}\chi_R & j=1,\\ \\
\displaystyle \frac{1}{4\pi 2^j}\frac{x}{2}\chi_R, & j\geq 2.
\end{array}\right.
\]
\end{example}

The above theorem can be used to construct uniformly bounded solutions
to \iref{diveq} for data $f$ in Lorentz spaces $Y=L^{d,q}(\RR^d)$ when
$1<q<\infty$. Indeed, since $L^{d,q}$ is reflexive, they qualify as Asplund spaces with \Frechet differentiable norm, see \cite{Asp1968} or \cite[thm 2.12]{Phe1993}. 
Except for the case $q=d$ corresponding to the space $Y=L^d$, the Fr\'echet 
derivative of the $L^{d,q}$ does not have a simple explicit form, however we note
that this is not required for defining the hierarchical solution.
Other applications of hierarchical constructions in inverse problems that arise in image processing can be found in \cite{MNR2019}.

In contrast, the space $Y=L^{d,\infty}(\RR^d)$ does not have a 
Fr\'echet differentiable norm, as seen by the following counter-example due Luc Tatar, \cite{Tar2011}.
The purpose is to show that there exists $f, g$ and $\alpha >0$ such that 
\begin{equation}\label{eq:notF}
\wkLp{f+\epsilon g} \geq \wkLp{f} + \alpha |\epsilon|
\end{equation}
proving  that $\wkLp{\cdot}$ is not \Frechet --- not even Gateaux differentiable. To this end one restrict attention to the unit interval $(0,1)$. Set $f(x):=x^{-\frac{1}{p}}$ 
and 
\[
g(x):=\sum_{k=0}^\infty (-1)^k2^{\frac{k}{p}}{\mathds 1}_{{\mathbb I}_k}(x), \qquad 
{\mathbb I}_k = (2^{-(k+1)},2^{-k}).
\]
 The second rearrangement of $f$ is given by  $\ds f^{**}(t)=\frac{p}{p-1}t^{-\frac{1}{p}}, \ 0<t<1$. Since $|g(x)|\leq f(x)$ it follows that $F:=f+\epsilon g \geq 0$ for $|\epsilon|<1$, and hence
\begin{equation}\label{eq:whynotF}
\begin{split}
\|f+\epsilon g\|_{\weakLp} & \geq t^{\frac{1}{p}}F^{**}(t)=t^{-1+\frac{1}{p}}\int_0^t F^*(s)\dds  \\ 
 & \geq  t^{-1+\frac{1}{p}}\int_0^t F(s)\dds
   = t^{-1+\frac{1}{p}}\int_0^t f^*(s)\dds +\epsilon t^{-1+\frac{1}{p}}\int_0^t g(s)\dds \\
 & =\|f\|_{\weakLp}+\epsilon t^{-1+\frac{1}{p}}\int_0^t g(s)\dds \ \ \textnormal{for all} \  \ t<1.
\end{split}
\end{equation}
It remains to lower bound the term on the right. We compute
\[
\int_{{\mathbb I}_k} g(s)\dds =(-1)^k2^{\frac{k}{p}}|{\mathbb I}_k|= (-1)^k\rho^k, \qquad \rho:=2^{-1+\frac{1}{p}} <1.
\]
It follows that $\ds \int_0^{2^{-k}} g(s)\dds = \frac{(-1)^k\rho^k}{1+\rho}$ implying that
$\ds \epsilon t^{-1+\frac{1}{p}}\int_0^t g(s)\dds_{\, |{\, t=2^{-k}}}= \frac{(-1)^k\epsilon}{1+\rho}$,
and  \eqref{eq:notF} follows from \eqref{eq:whynotF} at $t=2^{-k}$ with $(-1)^k\epsilon>0$ and $\ds \text{\large$\alpha=\nicefrac{1}{(1+\rho)}$}$.

\subsection{A hierarchical construction for general data.} 

Finally, we work with general data $f\in Y(\Omega)$ for some $\Omega\in\RR^d$,
without making assumption on the Fr\'echet differentiability of the $Y$ norm, 
but instead taking as a prior assumption that $Y$ is a space such that uniformly bounded
solution to \iref{diveq} exist with the bound  \iref{eq:divU=F} for some $\gamma>0$.

We also assuming that $Y$ is a dual space so that there exists a minimizer to the functional $F$ having the general form \iref{genfunc}. Here we use the exponent $p=1$ so that
\be
 \Fvee(\bu)=\Fvee_{\lambda}(\bu)=\|\bu\|_{L^\infty}+\lambda \|\div \bu-f\|_Y 
 \label{genfunc1}
 \ee
For a value of $\lambda$ to be fixed later, 
we denote by $\bu_1$ is the minimizer and $r_1=f-\div \u_1$ the residual.
In addition to the trivial bound
\be
\label{triv1}
\|r_1\|_{Y}\leq \|f\|_{Y} \quad{\rm and} \quad \|\bu_1\|_{L^\infty} \leq \lambda\|f\|_{Y},
\ee
that are obtained by comparison of $\Fvee(\u_1)$ and $\Fvee(0)$, we can also compare
$\Fvee(\u_1)$ with $\Fvee(\u)$ where $\u$ is a solution to \iref{diveq} 
satisfying the bound  \iref{eq:divU=F}. It follows that
\be
\lambda \|r_1\|_{Y} \leq \|\u_j\|_{L^\infty} \leq \gamma \|f\|_Y.
\ee
Therefore, taking $\lambda > \gamma$, we obtain a contraction property
\be
\|r_1\|_{Y} \leq \rho \|f\|_Y,
\ee
with $\rho=\gamma/\lambda<1$. 

This suggests a hierarchical construction that was
proposed in a more general context in \cite{Tad2014}:
with this fixed value of $\lambda$, we define iteratively 
$\u_j$ as the minimizer of 
\be
\Fvee_j(\bu)=\|\bu\|_{L^\infty}+\lambda \|r_{j-1}-\div \bu\|_Y,
\ee
and $r_j=r_{j-1}-\div \bu_j$. By recursive application of
the above contraction principle, it follows that
\be
\|r_{j}\|_Y \leq \rho \|r_{j-1}\|_Y \leq \cdots \leq  \rho^j \|f\|_Y.
\ee
as well as
\be
\|\u_j\| \leq \gamma \|r_{j-1}\|_Y \leq \gamma \rho^{j-1}\|f\|_Y.
\ee
From this it follows that the hierarchical construction
\be
\u_1+\u_2+\cdots+\u_j+\cdots,
\ee
converges to a uniformly bounded $\u$ that satisfies the equation \iref{diveq}
and the bound
\be
\|\u\|_{L^\infty} \leq \frac {\gamma}{1-\rho}\|f\|_Y.
\ee


\begin{thebibliography}{xx}

\ifx
\bibitem[AF2003]{AF2003} {\sc R. Adams and J. Fournier}, Sobolev Spaces, 2nd ed., Academic Press, 2003.
\fi

\bibitem[Aji2009]{Aji2009}
S. Ajiev, On Chebyshev centres, retractions, metric projections and homogeneous inverses for
Besov, Lizorkin-Triebel, Sobolev and other Banach spaces. 
East J. Approx. 15(3) (2009), 375-428.

\bibitem[Asp1968]{Asp1968}
E. Asplund, Fr\'{e}chet diﬀerentiability of convex functions, Acta Math. 121, (1968) 31-47.

\ifx
\bibitem[Ba1991]{Ba1991} {\sc K. Ball} {\sl Volume ratios and a reverse isoperimetric inequality}. J. London Math. Soc (1991)
	
\bibitem[BD1978]{BD1978} {\sc H. Berens and R. De Vore} {\sl Quantitative Korovkin Theorems for Positive Linear Operators on $L_p$- Spaces}, 
Trans. AMS  245, (1978), 349--361.
\fi

\bibitem[BC2011]{BC2011}
H. Bahouri and A. Cohen, Refined Sobolev inequalities in Lorentz spaces, J Fourier Anal Appl (2011) 17:662-673.

\bibitem[BS1988]{BS1988}
{\sc C. Bennett and R. Sharpley}, Interpolation of operators, Pure and Applied Mathematics, vol. 129, Academic Press, Orlando, 1988.

\bibitem[BL1976]{BL1976}
{\sc J. Bergh and J. L\"ofstr\"om}, Interpolation Spaces, An Introduction, 
Springer, 1976.

\bibitem[BB2003]{BB2003}
{\sc J. Bourgain and H. Brezis}, {\sl On the equation ${\rm div}\, Y=f$ and application to control of phases},
  J. Amer. Math. Soc.  16  (2003),  no. 2, 393--426.



\bibitem[BB2007]{BB2007}
{\sc J. Bourgain and H.  Brezis},  {\sl New estimates for elliptic equations and Hodge type systems}  J. Eur. Math. Soc. (JEMS)  9  (2007),  no. 2, 277--315.


\ifx
\bibitem[Cwi75]{Cwikel75}
 {\sc A. Cwikel}, {\sl On the dual of weak $L^p$}, 
Annales de l'institut Fourier 25 (2) (1975), 81--126.

\bibitem[ET1999]{ET1999}
{\sc I. Ekeland and R. T\'{e}mam}, Convex Analysis and Variational
Problems, SIAM, 1999.
\fi

\bibitem[EG1992]{EG1992}
{\sc L. Evans and R. Gariepy}, Measure Theory and Fine Properties of Functions, CRC Press, 1992.


\bibitem[Fed1969]{Fed1969}
{\sc H. Federer}, Geometric Measure Theory, Springer, 1969.


\bibitem[Mey2002]{Mey2002} {\sc Y. Meyer}, {\sl Oscillating Patterns in Image Processing and
        Nonlinear Evolution Equations}, University Lecture Series Volume 22, 
        AMS 2002.
        
\bibitem[MNR2019]{MNR2019} 
{\sc K. Modin, A. Nachman and  L. Rondic}
A multiscale theory for image registration and nonlinear inverse problems
Advances in Math. 346 (2019) 1009-1066.

\bibitem[PT2008]{PT2008}
N. C. Phuc and M. Torres, 
Characterizations of the existence and removable singularities
of divergence-measure vector fields, Indiana Univ. Math. J. 57 (2008), 1573-1597.

\bibitem[PT2017]{PT2017}
N. C. Phuc and M. Torres, Characterizations of signed measures in the dual of BV and related isometric isomorphisms. Annali della Scuola Normale Superiore di Pisa, Classe di Scienze (5). Vol. XVII (2017) 385-417

\bibitem[ONe1963]{ONe1963}
{\sc R. O'Neil}, {\sl Convolution operators and $L(p,q)$ spaces}, Duke
Math, J. 30, 129-143, 1963.


\bibitem[Phe1993]{Phe1993}
R. Phelps,
Convex Functions, Monotone Operators and Differentiability,
 Lecture Notes in Mathematics, Springer Verlag, 1993.


\bibitem[Roc1974]{Roc1974} {\sc R. T. Rockafeller}, Conjugate Duality and Optimization, SIAM/CBMS Monograph Ser. 16, SIAM Pub. 1974.


\bibitem[TT2011]{TT2011} {\sc E. Tadmor and C. Tan}, {\sl Hierarchical construction of bounded solutions of $div U=F$ in critical regularity spaces},  ``Nonlinear Partial Differential Equations'', Proc. 2010 Abel Symposium (H. Holden \& K. Karlsen eds.), Abel Symposia 7, Springer (2011), pp. 255-269.

\bibitem[Tad2014]{Tad2014}
{\sc E. Tadmor}, {\sl Hierarchical construction of bounded solutions in critical regularity spaces}, Communications in Pure \& Applied Mathematics 69(6) (2016) 1087-1109.

\bibitem[Tar189]{Tar1989}
{\sc L. Tartar}, {\sl Lorentz spaces and applications}, Lecture notes,
Carnegie Mellon University, 1989.

\bibitem[Tar2011]{Tar2011}
{\sc L. Tartar}, personal communication, 2011.

\bibitem[Zie1989]{Zie1989}
{\sc W. Ziemer}, Weakly Differentiable Functions, Graduate Texts in Math., Springer, vol. 120, 1989.
\end{thebibliography}
\end{document}